\documentclass[11pt, reqno,sumlimits]{article}
\usepackage{amssymb,amscd, epsfig, mathrsfs, xypic, amsmath,amsthm, color} 
\usepackage{hyperref}
\usepackage{blindtext}
\usepackage{enumerate}

\xyoption{all}
\newtheorem{thm}{Theorem}[section]  
\newtheorem{cor}[thm]{Corollary}
\newtheorem{lem}[thm]{Lemma}
\newtheorem{prop}[thm]{Proposition} 
\newtheorem{df-pr}[thm]{Definition-Proposition}

\theoremstyle{definition}
\newtheorem{defn}[thm]{Definition} 
 
\newtheorem{rem}[thm]{Remark}

\newcommand{\Sch}{{\operatorname{Sch}_{\bfk}}}

\newcommand{\ZZ}{{\mathbb Z}}

\newcommand{\PP}{{\mathbb P}}
\newcommand{\LL}{{\mathbb L}}

\newcommand{\sfn }{{\mathsf n}}

\newcommand{\sfp }{{\mathsf p}}

\newcommand{\sfs }{{\mathsf s}}

\newcommand{\bfk}{{\mathbf k }}

\newcommand{\calL}{{\mathcal L}}

\newcommand{\calO}{{\mathcal O}}
\newcommand{\calP}{{\mathcal P}}
\newcommand{\calQ}{{\mathcal Q}}

\newcommand{\calS}{{\mathcal S}}

\newcommand{\scB}{{\mathscr B}}
\newcommand{\tscB}{\widetilde{\mathscr B}}
\newcommand{\scC}{{\mathscr C}}
\newcommand{\tscC}{\widetilde{\mathscr C}}

\newcommand{\scP}{{\mathscr P}}

\newcommand{\scS}{{\mathscr S}}
\newcommand{\tscS}{\widetilde{\mathscr S}}

\newcommand{\lan}{{\langle}}
\newcommand{\ran}{{\rangle}}

\newcommand{\inc}{\hookrightarrow}

\newcommand{\Fl}{\operatorname{Fl}}

\newcommand{\Gr}{\operatorname{Gr}}

\newcommand{\gr}{\operatorname{gr}}

\newcommand{\Spec}{\operatorname{Spec}}

\newcommand{\rk}{{\operatorname{rk}}}

\newcommand{\id}{{\operatorname{id}}}

\newcommand{\supp}{{\operatorname{supp}}}

\newcommand{\OG}{{\operatorname{OG}}}
\newcommand{\SP}{{\calS\calP}}
\newcommand{\SM}{\mathbf{Sm}_{\bfk}}
\newcommand{\SCH}{\mathbf{Sch}_{\bfk}}
\newcommand{\LCI}{\mathbf{Lci}_{\bfk}}
\newcommand{\CK}{{\it CK}}

\newcommand{\SG}{{\it SG}}

\newcommand{\srarrow}{\twoheadrightarrow}
\newcommand{\irarrow}{\hookrightarrow}

\newcommand{\Laz}{\mathbb{L}}
\newcommand{\trecd}{\cdot\cdot\cdot}

\newcommand{\spec}{{\rm Spec\,}}

\newcommand{\bfs}{{\mathbf s }}

\newsavebox{\savepar}

\numberwithin{equation}{section}

\newcounter{labelflag} 
\newcommand{\labelon}{\setcounter{labelflag}{1}}
\newcommand{\Label}[1]{\ifnum\thelabelflag=1\ifmmode
\makebox[0in][l]{\qquad\fbox{\rm#1}} \else
\marginpar{\vspace{0.7\baselineskip} \hspace{-1.1\textwidth}
\fbox{\rm#1}} \fi \fi \label{#1} } \labelon

\begin{document} 
\title{Symplectic and odd orthogonal Pfaffian formulas for algebraic cobordism}
\author{Thomas Hudson, Tomoo Matsumura}
\date{}
\maketitle 
\begin{abstract} 
In this paper, we prove generalizations of Pfaffian formulas for the degeneracy loci classes in the algebraic cobordism of the symplectic/odd orthogonal Grassmann bundles.
\end{abstract}

\section{Introduction}
The $r$-th \textit{degeneracy locus} of a morphism of vector bundles $\varphi: E\rightarrow F$ is the subvariety $D_r(\varphi)$ of the base scheme $X$ formed by all the points $x$ over which the image of  the restriction $\varphi(x)$ has dimension at most $r$. If, for simplicity, one assumes $X$ to be a smooth quasi-projective scheme and $\varphi$ to be sufficiently general, the classical Giambelli--Thom--Porteous formula  describes the Chow ring fundamental class $[D_r(\varphi)]_{CH}$ as a determinant in Chern classes. Similarly, it is possible to consider more restrictive settings in which $\varphi$ is either \textit{skewsymmetric} or \textit{symmetric}, which means that one has $F=E^\vee$ and $\varphi^\vee=-\varphi$ (respectively $\varphi=\varphi^\vee$). In both cases $[D_r(\varphi)]_{CH}$ is given by a Pfaffian formula in the Chern classes of $E$.

 A more general family of degeneracy loci can be constructed by replacing $E$ and $F$ with flags of vector bundles $E_{p_1}\subseteq\trecd \subseteq E_{p_l}=E$ and $F=F_{q_m}\srarrow\trecd \srarrow F_{q_1}$, by considering the morphisms $\varphi_{i,j}:E_i\rightarrow F_j$ induced by $\varphi: E\rightarrow F$ and taking the intersection of the loci $D_{r_{i,j}}(\varphi_{i,j})$. Fundamental examples of these loci are represented by the Schubert varieties of Grassmannians or, more generally, of partial flag varieties. Just as before, it is possible to take into consideration ``skewsymmetric'' and ``symmetric'' variants. Instead of subspaces of a fixed dimension inside the given vector space, the symplectic Grassmannian  (respectively orthogonal Grassmannian) consists of subspaces which are \textit{isotropic}, \textit{i.e.} on which a given symplectic (respectively symmetric bilinear) form vanishes identically. Inside of these ambient spaces one can construct Schubert varieties which are indexed by the combinatorial objects known as $k$-\textit{strict} partitions.
 
 In \cite{PragaczPQ} Pragacz considered the case in which the dimension of the isotropic subspaces is maximal and showed that all Schubert classes, the fundamental classes of Schubert varieties, can be expressed through a Pfaffian formula. Later, in \cite{Kazarian}, Kazarian generalised Pragacz's formula to the case of isotropic Grassmann bundles in which the vector space is replaced by a vector bundle $E\rightarrow X$. An alternative proof for maximal symplectic Grassmann bundles, also known as Lagrangian, was given by Ikeda \cite{Ikeda} by means of equivariant cohomology. The formulas for the non-maximal Grassmannians are instead due to Buch--Kresh--Tamvakis \cite{BuchKreschTamvakis1} and in \cite{WilsonThesis} Wilson made a conjecture for the explicit formula describing the Schubert classes of non-maximal symplectic Grassmann bundles, which was later proved by Ikeda and the second author \cite{IkedaMatsumura}. Alternative proofs of the conjecture were given by Tamvakis--Wilson \cite{TamvakisWilson} and Anderson--Fulton \cite{AndersonFulton2}.

 In recent years, following the trend of generalised Schubert calculus, there has been an attempt to lift classical results in $CH^*$ to other functors like connective $K$-theory and algebraic cobordism, respectively denoted $CK^*$ and $\Omega^*$. In \cite{HIMN}, together with T. Ikeda and H. Naruse, we established Pfaffian formulas describing the $K$-theoretic Schubert classes of symplectic and odd orthogonal Grassmann bundles which generalise those that have already been mentioned. The goal of this paper is to further extend these formulas to $\Omega^*$.
 
A key aspect of algebraic cobordism, which was established by Levine--Morel in \cite{LevineMorel}, is its universality. In particular, this means that formulas which hold for $\Omega^*$ can be specialised to any other oriented cohomology theory in exactly the same way in which expressions in $CK^*$ recover the classical ones. An easy example of this phenomenon is illustrated by the behaviour of the first Chern class of line bundles. In $CH^*$ one has 

 \vspace{-4 mm}

$$c_1(L\otimes M)=c_1(L)+c_1(M)\ \text{ and\  }\ c_1(L^\vee)=-c_1(L),$$

 while in $CK^*$ these equalities become 
 
 \vspace{-4.5 mm}
 
$$c_1(L\otimes M)=c_1(L)+c_1(M)-\beta c_1(L)c_1(M) \ \text{ and }\ c_1(L^\vee)=\frac{-c_1(L)}{1-\beta c_1(L)},$$

 with $\beta\in CK^*(\spec \bfk)$ being the pushforward of the fundamental class of $\PP^1$ to the point. In order to describe $c_1(L\otimes M)$ and $c_1(L^\vee)$ in a general theory $A^*$, it becomes necessary to introduce $F_A$ and $\chi_A$, two power series in respectively two and one variables which are known as \textit{formal group law} and \textit{formal inverse}.
 
 Let us begin by explaining our results in the case of a symplectic Grassmann bundle $\SG^k (E)$, where $E\rightarrow X$ is a vector bundle of rank $2n$ with a nowhere vanishing skewsymmetric form and for every $x\in X$ the fiber $\SG^k (E)_x$ is the symplectic Grassmannian of $(n-k)$-dimensional isotropic subspaces of $E_x$. Exactly as Kazarian did in \cite{Kazarian}, in \cite{HIMN} for every $k$-strict partition $\lambda$ the $K$-theoretic fundamental class of the Schubert variety $X_\lambda^C$ was obtained by computing $\psi_*[Y_\lambda^C]_{CK}$, the pushforward of the fundamental class of a resolution of singularities $\psi:Y_\lambda^C\rightarrow X_\lambda^C\irarrow \SG^k E$. The final formula was expressed in terms of the relative Segre classes $\scS_i^{CK}\big((U-E/F^j)^\vee\big)$, some characteristic classes associated to the tautological isotropic subbundle $U$ and to the elements of the given reference flag of subbundles of $E$ used to define the Schubert varieties
 $$0=F^n\subset F^{n-1} \subset \cdots \subset F^1 \subset F^0 \subset F^{-1} \subset \cdots \subset F^{-n}=E.$$
 
 Since in a general oriented cohomology theory not all Schubert varieties have a well defined notion of fundamental class, we opted for $\psi_*[Y_\lambda^C]_A$ as a replacement and, by combining Kazarian's approach with the generalised relative Segre classes $\scS^A$ introduced in \cite{HudsonMatsumura}, we obtained the following description.     

\vspace{2 mm}

\noindent{\bf Main Theorem }(Theorem \ref{detthmC}). {\it 
 Let $A^*$ be any oriented cohomology theory and for each $m\in \ZZ$ and $-n\leq \ell \leq n$ set $\scC_m^{(\ell)} := \scS_{m}^A\big(U^\vee-(E/F^{\ell})^{\vee})\big)$. Then, for every $k$-strict partition $(\lambda_1,\trecd,\lambda_r)$,  we have the following equality in $A^*(\SG^k(E))$:
\begin{equation}\label{Theorem I}
[Y_\lambda^C \to \SG^k(E)]_A:=\psi_*[Y_\lambda^C]_A= \sum_{\sfs=(s_1,\dots, s_r) \in\ZZ^r}c_{\sfs}^{\lambda}\scC_{\lambda_1+s_1}^{(\chi_1)} \cdots \scC_{\lambda_r+s_r}^{(\chi_r)}.
\end{equation}
Here $\chi=(\chi_1,\chi_2,\trecd, \chi_r)$ is the \text{characteristic index} defined in Section \ref{seckstrict} and the coefficients $c_{\sfs}^\lambda\in A^*(\spec \bfk)$ are given by the Laurent series expansion
\begin{equation*}
\frac{\prod_{1\leq i<j\leq r}(1-t_i/t_j)P_A(t_j,t_i)}{\prod_{(i,j)\in C(\lambda)}(1-\bar t_i/t_j)P_A(t_j, \bar t_i)}=\sum_{\sfs=(s_1,\dots, s_r) \in\ZZ^r}c_{\sfs}^{\lambda}\cdot t_1^{s_1}\cdots t_r^{s_r},
\end{equation*}
where $C(\lambda):=\{(i,j) \ |\ 1\leq  i < j, \ \ \chi_i + \chi_j \geq 0\}$, $\bar t_i=\chi_A(t_i)$ and $P_A$ is the unique power series satisfying $F_A(u,\chi_A(v))=(u-v)P_A(u,v)$.
}

\vspace{3 mm}
It is worth pointing out that when $\lambda=(\lambda_1)$ one simply has $[Y^C_{\lambda_1}\rightarrow \SG^k(E)]_A=\scC^{(\chi_1)}_{\lambda_1}$, since the Laurent series defining the coefficients reduces to 1. In particular this means that it consists of a single relative Segre class. This is in general not the case when one considers an odd orthogonal Grassmann bundle $\OG^k (E)$, with $E$ of rank $2n+1$ and each fiber being an orthogonal Grassmannian of $(n-k)$-dimensional isotropic subspaces. The essential difference with the previous situation is that it is far more complex to deal with the case of quadric bundles $\OG^{n-1}(E)=Q(E)$, the orthogonal analogue of projective bundles. In fact, if we set the reference flag to be 
$$0=F^n\subset F^{n-1} \subset \cdots \subset F^1 \subset F^0 \subset (F^0)^{\perp} \subset F^{-1} \subset \cdots \subset F^{-n}=E,
$$
then, as elements of $A^*(Q(E))\otimes_\ZZ \ZZ[1/2]$, the Schubert classes are given by
\[
[X^B_{\lambda_1} \to Q(E)]_A =\scB_{\lambda_1}^{(\chi_1)}:= \begin{cases}
\scS^A_{\lambda_1}\big(U^\vee- (E/F^{\chi_1})^{\vee}\big) & (0\leq \lambda_1 < n)\\
 \displaystyle\frac{1}{F_A^{(2)}\big(c_1(U^\vee)\big)} \scS^A_{\lambda_1}\big(U^\vee- (E/F^{\chi_1})^{\vee}\big) & (n\leq  \lambda_1 < 2n),
\end{cases}
\]
where $F_A^{(2)}(u)$  is the power series defined by the equation $F^A(u,u) = u \cdot F_A^{(2)}(u)$. More generally, the pushforward classes $[Y_\lambda^B\rightarrow \OG^k(E)]_A\in A^*(\OG^k(E))\otimes_\ZZ \ZZ[1/2]$ associated to the Schubert varieties $X_\lambda^B$ are obtained from (\ref{Theorem I}) by replacing $\scC_m^{(i)}$ with $\scB_m^{(i)}$ (see Theorem \ref{MainThmB}).

Recently, in \cite{K-Anderson}, Anderson extended the results of \cite{HIMN} to more general degeneracy loci including those arising from  even orthogonal Grassmann bundles. His work is based on the approach he and Fulton employed in their study of the Chow ring fundamental classes of degeneracy loci for all types \cite{AndersonFulton,AndersonFulton2}. In our future work we would like to lift Anderson's results to $\Omega^*$ so to cover also the even orthogonal case.  

The organisation of the paper is as follows. In section 2 we recall some basic facts about Borel--Moore homology theories, the covariant analogue of oriented cohomology theories, and we translate into this setting the results on Segre classes presented in \cite{HudsonMatsumura}. In section 3 we prove the main theorem for symplectic Grassmann bundles, while in section 4 we first deal with the special case of quadric bundles, which we then use to establish the main theorem for odd orthogonal Grassmann bundles.

\paragraph*{Notations and conventions:} Throughout this paper $\bfk$ will be a field of characteristic 0. By $\SCH$ we will denote the category of separated schemes of finite type over $\bfk$ and $\LCI$  will stand for its full subcategory constituted by the objects whose structural morphism is a local complete intersection. For a given category $C$ we will write $C'$ to refer to its subcategory given by allowing only projective morphisms. $\textbf{Ab}_*$ represents the category of graded abelian groups.

\section{Preliminaries}
The goal of this section is to collect some basic properties of Borel--Moore homology theories and to translate in this context some of the results on generalised Segre classes appeared in \cite{HudsonMatsumura}.
\subsection{Borel--Moore homology theories}

An oriented Borel--Moore homology theory on  $\SCH$ (or \textit{mutatis mutandis} on $\LCI$) is given by a covariant functor $A_*: \SCH'\rightarrow \mathbf{Ab}_*$ , by a family of pullback maps $\{f^*:A_*(Y)\rightarrow A_*(X)\}$ associated to \textit{l.c.i.} morphism and by an external product $A_*(X)\otimes A_*(Y)\rightarrow A_*(X\times_{\spec\,k} Y)$. Let us remind the reader that a morphism is a \textit{local complete intersection} if and only if it can be factored as the composition of a regular embedding and a smooth morphism. A detailed description of the properties that these three components have to satisfy would force us to take a significant detour, so we will focus only on the aspects that are more relevant to our work and refer the reader to \cite[Definition 5.1.3]{LevineMorel} for the precise definition. 

For us the most relevant feature of oriented Borel--Moore homology theories is that they satisfy the projective bundle formula. Roughly speaking it states that for every vector bundle $E$ of rank $e$ with $X\in \SCH$, the evaluation of $A_*$ on the associated dual projective bundle $\PP^*(E)\stackrel{q}\rightarrow X$ can be described in terms of $A_*(X)$. More precisely for $i\in\{0,1,\dots,e-1\}$ one has operations
$$\xi^{(i)}: A_{*+i-e+1}(X)\longrightarrow A_*(\PP^*(E))$$
given by $\xi^{(i)}:= \tilde{c}_1(\calQ)^i\circ q^*$, where $\calQ\rightarrow \PP^*(E)$ is the tautological line bundle and $\tilde{c}_1(\calQ):=s^*\circ s_*$, for any section $s:\PP^*(E)\rightarrow \calQ$. Altogether these yield the following isomorphism
$$\varPsi:\bigoplus_{i=0}^{e-1}A_{*+i-e+1}(X)\stackrel{\sum_{i=0}^{e-1}\xi^{(i)}}\longrightarrow A_*(\PP^*(E)).$$ 
A very important consequence of this is that every oriented Borel--Moore homology theory admits a theory of Chern class operators: to $E$ one associates $\{\tilde{c}^A_i(E):A_*(X)\rightarrow A_{*-i}(X)\}_{0\leq i \leq e}$. These are defined by setting $\tilde{c}^A_0(E)=\text{id}_{A_*(X)}$ and, up to a sign, by considering the different components of $\varPsi^{-1}\circ\xi^{(e)}$, so that one obtains the relation 
$$\sum_{i=0}^e (-1)^i\xi^{(e-i)}\circ \tilde{c}^A_i(E)=0.$$
These operators can be collected in the so-called \textit{Chern polynomial} $\tilde{c}^A(E;u):=\sum_{i=0}^e \tilde{c}^A_i(E)u^i$ and it is worth mentioning that, in view of the Whitney formula, its definition can be extended to the Grothendieck group of vector bundles by setting 
$$\tilde{c}^A(E-F;u):=\frac{\tilde{c}^A(E;u)}{\tilde{c}^A(F;u)}.$$   

Beside being extremely useful for computations, Chern classes allow one to get some insight on how a general oriented Borel--Moore homology theory $A_*$ differs from the Chow group $CH_*$, probably the most commonly known example. Let us consider, as an example, the behaviour of the first Chern class with respect to the tensor product of two line bundles $L$ and $M$. While in $CH_*$ one has 
$$\tilde{c}^{CH}_1(L\otimes M)=\tilde{c}^{CH}_1(L)+\tilde{c}^{CH}_1(M),$$
in general the relation between the three Chern class operators is described by a \textit{formal group law} $\big(A_*(\spec k),F_A\big)$, where $F_A(u,v)$ is a special power series with coefficients in the coefficient ring of the theory $A_*(\spec k)$. The precise relation is given by
$$\tilde{c}^{A}_1(L\otimes M)=F_A\left(\tilde{c}^A_1(L),\tilde{c}^A_1(M)\right).$$    
In a similar fashion, whereas in $CH_*$ one simply has $\tilde{c}_1^{CH}(L^\vee)=-\tilde{c}_1^{CH}(L)$, in general one needs to introduce the \textit{formal inverse} $\chi_A$, a power series in one variable satisfying both
$$\tilde{c}_1^{A}(L^\vee)=\chi_A\big(\tilde{c}_1(L)\big) \text{\quad and\quad } F_A\big(u,\chi_A(u)\big)=0.$$
In some case we will denote the formal inverse $\chi_A(u)$ simply by $\overline{u}$.

All our computations will take place in the algebraic cobordism of Levine--Morel $\Omega_*$ and our choice is motivated by the following fundamental result.

\begin{thm}[\protect{\cite[Theorems 7.1.3 and 4.3.7]{LevineMorel}}]
$\Omega_*$ is universal among oriented cohomology theories on $\LCI$. That is, for any other oriented Borel-Moore homology theory $A_*$ there exists a unique morphism 
$$\vartheta_A:\Omega_*\rightarrow A_*$$
of oriented Borel--Moore homology theories. Furthermore, its associated formal group law $\big(\Omega_*(\Spec \bfk),F_\Omega\big)$ is isomorphic to the universal one defined on the Lazard ring $(\Laz,F)$.
\end{thm} 

One consequence of this universality is that all the formulas obtained for $\Omega_*$ can be specialised to every other oriented Borel--Moore homology theory $A_*$. In other words, algebraic cobordism allows one to work with all theories at once. Since we will only work with algebraic cobordism, in the remainder of the paper we will remove the subscript $_\Omega$ from the notation.

Let us conclude our general discussion by briefly mentioning the construction of fundamental classes and some results which can be used to compute them. To every $X\in\SCH$ whose structural morphism $\pi_X$ is l.c.i. we associate its fundamental class by setting $1_X:=\pi_X^*(1)$. Notice that here $1$ stands for the multiplicative unit in $A_*(\spec k)$. 
In the special case of the zero scheme of a bundle, the fundamental class can be computed via the following lemma.

\begin{lem}[Lemma 6.6.7 \cite{LevineMorel}]\label{lemCKGB}
Let $E$ be a vector bundle of rank $e$ over $X \in \SCH$. Suppose that $E$ has a section $s: X \to E$ such that the zero scheme of $s$, $i:Z\to X$ is a regularly embedded closed subscheme of codimension $d$. Then we have
\[
\tilde{c}_e(E) = i_*\circ i^*.
\]
In particular, if $X$ is an l.c.i. scheme, we have 
\[
\tilde{c}_e(E)(1_X) = i_*(1_Z).
\]
\end{lem}

Finally, as it will play an important role in our computations, we would like to make more explicit the case of the fundamental class of a nonreduced divisor. For this we will require a bit of notation. For every integer $n\geq2$, let $n\cdot_{F_A} u$ be the \textit{formal multiplication} by $n$, \textit{i.e} the power series obtained by adding $n$ times the variable $u$ using the formal group law $F_A$. Since $F_A$ is a formal group law, one has 
\begin{align}\label{formalmult}
n\cdot_{F_A}(u)=u\cdot F^{(n)}_A(u)
\end{align}
 for some degree 0 power series $F^{(n)}_A(u)$ whose costant term is $n$. We are now able to restate \cite[Proposition 7.2.2]{LevineMorel} for the particular case we will need.
 

\begin{lem}\label{lem div}
Let $W$ be a smooth scheme and $D$ a smooth prime divisor of $W$. For any integer $n\geq2$, let $|E|$ be the closed subscheme associated to the divisor $E=nD$. If $L$ is the line bundle corresponding to $D$ and $\iota:D\rightarrow |E|$ is the natural morphism, then in $A_*(|E|)$ we have 
$$1_{|E|}=\iota_*\Big(F_A^{(n)}\big(\tilde{c}^A_1(L_{|D})\big)(1_D)\Big),$$
where $L_{|D}$ is the restriction of $L$ to $D$.
\end{lem}

\subsection{Segre operators}
In \cite{HudsonMatsumura}, in order to be able to describe the pushforwards along projective bundles, we generalised to algebraic cobordism the classical definition of Segre classes given in \cite{IntersectionFulton} by Fulton. In the context of oriented Borel--Moore homology theories we can rephrase it as follows.

\begin{defn}\label{defSegre}
Let $E$ be a vector bundle of rank $e$ over $X$. For each $m \in \ZZ$, consider the dual projective bundle $\pi: \PP^*(E\oplus O_X^{\oplus n}) \to X$ where $O_X$ is the trivial line bundle over $X$. Define the degree $m$ \emph{Segre class operatos} $\widetilde{\scS}_m(E)$ of $E$ by 
\[
\widetilde{\scS}_m(E) = \pi_*\circ\tilde{c}_1(\calQ)^{m+e+n-1}\circ \pi^*, 
\]
where  $\calQ$ is the first Chern class of the tautological quotient line bundle of $\PP^*(E\oplus \calO_X^{\oplus n})$. We set
\[
\widetilde{\scS}(E;u) := \sum_{m\in \ZZ} \widetilde{\scS}_m(E)u^m.
\]
\end{defn}

\begin{rem}
It is worth pointing out that the definition of Segre classes does not depend on the choice of $n$: any increase of this parameter will leave $\widetilde{\scS}_m(E)$ unchanged.
\end{rem}

\begin{rem}
If $E$ is a line bundle and $n=0$, then one has $\calQ=E$ and $\pi=\id_X$. In particular this implies that $\widetilde{\scS}_m(E)=\tilde{c}_1(E)^m$ for all $m\geq 0$.
\end{rem}

It is actually possible to relate these operators to the Chern classes as in \cite[Theorem 3.6]{HudsonMatsumura}, but for this we first need to define the following power series.
 
\begin{defn}\label{defwtilde1}
Let $\{x_1,\dots, x_e\}$ be the Chern roots of a vector bundles $E$ and define $\tilde{w}_{-s}(E)$ 
by setting
\[
\prod_{i=1}^e P(z, x_{i}) =\sum_{s\geq 0}^{\infty} \tilde{w}_{-s}(E) z^{s},
\]
where $P$ is the unique power series satisfying the equality $F(z,\chi(x)) = (z-x) P(z,x)$. Since this expression is symmetric it is independent of the chosen roots and we can also set $\tilde{w}(E;u):=\sum_{s\geq 0}^{\infty} \tilde{w}_{-s}(E)u^{-s}$. It should be noticed that $\tilde{w}_0(E)$ has constant term $1$ and as a consequence $\tilde{w}(E;u)$ is an invertible power series. 
\end{defn}

\begin{prop}\label{thmSegre}
Let $E\rightarrow X$ be a vector bundle of rank $e$ over $X\in \SCH$. Then we have the following equality of power series:
\[
\widetilde{\scS}(E;u) =\frac{\scP(u)}{\tilde{c}(E;-u)\tilde{w}(E;u)}.
\]
Here $\scP(u):=\sum_{i=0}^{\infty} [\PP^i]u^{-i}$ is the power series collecting the pushforwards classes of projective spaces $[\PP^i]:=[\PP^i\rightarrow \spec k]\in\Laz^{-i}$.
\end{prop}

This relation can be used to lift the notion of Segre classes to the Grothendieck group of vector bundles (\cite[Definition 3.8]{HudsonMatsumura}). In fact, it easily follows from Definition \ref{defwtilde1} that $\tilde{w}(-\,;u)$ is additive on exact sequences and so one can meaningfully evaluate such collection of operators on virtual bundles $[E]-[F]$ by setting  $\tilde{w}(E-F;u):=\tilde{w}(E;u)\circ \tilde{w}^{-1}(F;u)$. 

\begin{defn}\label{relSeg}
Let $X\in \Sch$. For vector bundles $E$ and $F$ over $X$, respectively of rank $e$ and $f$, define the \textit{relative Segre class operators} $\tscS_m(E-F)$ on $\Omega_*(X)$ as
\begin{equation}\label{dfSegX1}
\tscS(E-F;u):=\sum_{m\in \ZZ} \tscS_m(E-F) u^m = \scP(u)\frac{\tilde{c}(F;-u)\circ \tilde{w}(F;u)}{\tilde{c}(E;-u)\circ\tilde{w}(E;u)}
\end{equation}
or, equivalently, by setting
\begin{equation}\label{eqRelSeg1}
\tscS_m(E-F) :=\sum_{p=0}^{\infty}\sum_{q = 0}^{\infty}  (-1)^p\tilde{c}_p(F)\circ\tilde{w}_{-q}(F)\circ\tscS_{m-p+q}(E).
\end{equation}
\end{defn}

\begin{rem}
It is easy to check that the last equality holds even if one allows $F$ to be a virtual bundle.  
\end{rem}

\begin{rem}\label{remLINE}
In the special case in which $E$ is a line bundle one has $\pi = \id_X$ and as a consequence $\tilde{c}_1(\calQ)^s\circ \tilde{c}_f(\calQ \otimes F^{\vee}))  = \widetilde{\scS}_{f-e+1+s}(E-F)$.
\end{rem}

We conclude this section by providing a description of relative Segre classes in terms of pushforwards of Chern classes. This is should be seen as an analogue of \cite[Theorem 3.9]{HudsonMatsumura}.

\begin{thm}\label{thmlcipush}
Let $X \in \SM$ and $i: Z \inc X$ a regular embedding. Let $E$ be a vector bundle of rank $e$ over $X$ and $F$ a vector bundle of rank $f$ over $Z$. Let $\pi': \PP^*(E|_Z) \to Z$ be the dual projective bundle of the restriction $E|_Z$ of $E$ and $Q_Z$ its universal quotient line bundle. The pullback of $F$ to $\PP^*(E|_Z)$ is denoted also by $F$. In $\Omega_*(Z)$, we have
\[
\pi'_*\circ \tilde{c}_1(\calQ_Z)^s \circ \tilde{c}_f(\calQ_Z \otimes F^{\vee})(1_{\PP^*(E|_Z)}) = \tscS_{s+f - e+1}(E|_Z - F)(1_Z).
\]
\end{thm}
\begin{proof}
Let us begin by observing that, since an easy Chern roots computation analogue to \cite[Formula (3.1)]{HudsonMatsumura} gives us 
$$\tilde{c}_f(\calQ_Z \otimes F^{\vee})=\sum_{p=0}^f \sum_{j=0}^{\infty}(-1)^p\tilde{c}_p(F)  \circ \tilde{w}_{-q}(F)\circ \tilde{c}_1(\calQ_Z)^{f-p+q},$$
the left hand side can be rewritten as
\begin{eqnarray}\label{eq push}
\sum_{p=0}^{f} \sum_{q=0}^{\infty} (-1)^p\tilde{c}_p(F)  \circ \tilde{w}_{-q}(F)\circ \pi'_*\circ \tilde{c}_1(\calQ_Z)^{s+f-p+q}(1_{\PP^*(E|_Z)}).
\end{eqnarray}
 We now notice that $1_{\PP^*(E|_Z)}$ can be rewritten as $\pi'^*1_Z$ so to obtain the Segre classes of $E|_{Z}$. Then (\ref{eq push}) becomes  
\begin{eqnarray*}
\sum_{p=0}^{f} \sum_{q=0}^{\infty} (-1)^p\tilde{c}_p(F)  \circ \tilde{w}_{-q}(F)\circ \tscS_{s+f-e+1-p+q}(E|_Z)(1_Z).
\end{eqnarray*}
 The right hand side of the statement can be now obtained by using the definition of relative Segre classes (\ref{eqRelSeg1}).
\end{proof}
\ 
\section{Symplectic degeneracy loci}
For this section we fix a nonnegative integer $k$.
\subsection{$k$-strict partitions and characteristic indices}\label{seckstrict}
A $k$-strict partition $\lambda$ is a weakly decreasing infinite sequence $\lambda_1,\lambda_2,\dots$ such that the number of nonzero parts is finite, and if $\lambda_i > k$, then $\lambda_i>\lambda_{i+1}$. The length of $\lambda$ is the number of nonzero parts $\lambda$. Let $\SP^k$ be the set of all $k$-strict partition.  Let $\SP_r^{k}$ be the set of all $k$-strict partitions with the length at most $r$. If $\lambda \in \SP_r^{k}$, then we write $\lambda=(\lambda_1,\dots,\lambda_r)$. Let  $\SP^k(n)$ be the set of all $k$-strict partitions such that $\lambda_1\leq n+k$ and the length of $\lambda$ is at most $n-k$.

Let $W_{\infty}$ be the infinite hyperoctahedral group which can be identified with the group of all signed permutation, \textit{i.e.}, all permutations $w$ of $\ZZ \backslash \{0\}$ such that $w(i)\not= i$ for only finitely many $i\in \ZZ \backslash \{0\}$, and $\overline{w(i)}=w(\bar{i})$ for all $i$ where $\bar{i}:=-i$. A signed permutation $w$ is determined by the sequence $(w(1),w(2),\dots)$ which we call one line notation. An element $w \in  W_{\infty}$ is called $k$-Grassmannian if 
\[
0<w(1) <\cdots < w(k), \ \ w(k+1)<w(k+2)<\dots.
\]
The set of all $k$-Grassmannian elements in $W_{\infty}$ is denoted by $W_{\infty}^{(k)}$. 

There is a bijection between $W_{\infty}^{(k)}$ and $\SP^k$ defined as follows. For each $w\in W_{\infty}^{(k)}$, the corresponding $k$-strict partition is given by
\[
\lambda_i :=\begin{cases}
w(k+i) & \mbox{ if $w(k+i)<0$}\\
\#\{ j\leq k \ |\ w(j) > w(k+i) \} & \mbox{ if $w(k+i)>0$}.
\end{cases}
\]
For each $\lambda$ (and the corresponding signed permutation $w$), we define its characteristic index $\chi=(\chi_1,\chi_2,\dots)$ by
\[
\chi_i := \begin{cases}
-w(k+i)-1 & \mbox{if $w(k+i) <0$}\\
-w(k+i) & \mbox{if $w(k+i)>0$}. 
\end{cases}
\]
Moreover, the following notations are necessary for our formulas of Grassmannian degeneracy loci in type C and B: for each $\lambda \in \SP^k$ and the corresponding characteristic index $\chi$, define
\begin{eqnarray*}
C(\lambda)&:=&\{(i,j) \ |\ 1\leq  i < j, \ \ \chi_i + \chi_j \geq 0\}, \\
\gamma_j&:=& \sharp\{ i \ |\ 1\leq  i < j, \ \ \chi_i + \chi_j \geq 0\}\ \ \ \ \mbox{for each $j>0$}.
\end{eqnarray*}

\subsection{Symplectic degeneracy loci}\label{ssec: Lag deg}
Let $E$ be a symplectic vector bundle over $X$ of rank $2n$, \textit{i.e.}, we are given a nowhere degenerating section of $\wedge^2 E$. For a subbundle $F$ of $E$, we denote by $F^{\perp}$  the orthogonal complement of $F$ with respect to the symplectic form. Fix a reference flag $F^{\bullet}$ of subbundles of $E$
\[
0=F^n\subset F^{n-1} \subset \cdots \subset F^1 \subset F^0 \subset F^{-1} \subset \cdots \subset F^{-n}=E,  
\]
where  $\rk\, F^i=n-i$ and $(F^i)^{\perp} = F^{-i}$ for all $i$ with $-n \leq i \leq n$. Let $\SG^k(E) \to X$ be the Grassmannian bundle over $X$ consisting of pairs $(x,U_x)$ where $x \in X$ and $U_x$ is an $n-k$ dimensional isotropic subspace of $E_x$. Let $U$ be the tautological bundle of $\SG^k(E)$. 

For each $\lambda \in \SP^k(n)$ of length $r$, let $X_{\lambda}^C$ be the symplectic degeneracy locus in $\SG^k(E)$ defined by
\[
X_{\lambda}^C = \Big\{ (x, U_x)\in \SG^{k}(E) \ |\ \dim (U_x \cap F^{\chi_i}_x) \geq i, \ \ i=1,\dots, r\Big\}
\]
where $\chi=(\chi_1,\chi_2,\dots)$ is the characteristic index for $\lambda$.

Let $\Fl_r(U) \to \SG^k(E)$ be the $r$-step flag bundle of $U$ where the fiber at $(x,U_x) \in \SG^k(E)$ consists of the flag $(D_{\bullet})_x=\{(D_1)_x \subset \cdots (D_r)_x\}$ of subspaces of $U$ with $\dim (D_i)_x = i$. Let $D_1 \subset \cdots \subset D_r$ be the flag of tautological bundles of $\Fl_r(U)$.  We set $D_0=0$. The bundle $\Fl_r(U)$ can be constructed as a tower of projective bundles
\begin{eqnarray}\label{towerC}
&&\pi: \Fl_r(U)=\PP(U/D_{r-1}) \stackrel{\pi_r}{\longrightarrow} \PP(U/D_{r-2}) \stackrel{\pi_{r-1}}{\longrightarrow} \cdots \ \ \ \ \ \ \ \ \ \ \ \ \ \ \ \ \ \ \ \ \ \ \ \ \nonumber\\\label{C P tower}
&&\ \ \ \ \ \ \ \ \ \ \ \ \ \ \ \ \ \ \ \ \ \ \ \ \ \ \ \ \ \ \ \ \ \ \ \ \ \ \ \ \ \ \ \ \ \ \ \  \stackrel{\pi_3}{\longrightarrow} \PP(U/D_1) \stackrel{\pi_2}{\longrightarrow} \PP(U)  \stackrel{\pi_1}{\longrightarrow} \SG^k(E).
\end{eqnarray}
The quotient line bundle $D_j/D_{j-1}$ is regarded as the tautological line bundle of $\PP(U/D_{j-1})$ and we set $\tau_j:=c_1((D_j/D_{j-1})^{\vee})$. We are now able to define the resolution of singularities of the degeneracy loci.
\begin{defn}
For each $j=1,\dots,r$, we define a subvariety $Y_j$ of $\PP(U/D_{j-1})$ by 
\[
Y_j := \Big\{ (x,U_x, (D_1)_x, \dots, (D_j)_x) \in \PP(U/D_{j-1}) \ |\ (D_i)_x \subset F^{\chi_i}_x, \ i=1,\dots,{j}\Big\}.
\]
We set $Y_0:=\SG^k_r(U)$ and $Y_{\lambda}^C:=Y_r$. Let $P_{j-1}:=\pi_j^{-1}(Y_{j-1})$, $\pi_j': P_{j-1} \to Y_{j-1}$ the projection and $\iota_j: Y_j \to P_{j-1}$ the obvious inclusion. Let $\sfp_j:=\pi_j'\circ\iota_j$. We have the commutative diagram
\[
\xymatrix{
\PP(U/D_{j-1}) \ar[r]_{\pi_j} & \PP(U/D_{j-2})\\
P_{j-1} \ar[r]_{\pi_j'}\ar[u] & Y_{j-1}\ar[u]\\
Y_j\ar[u]_{\iota_j}\ar[ru]_{\sfp_j} &
}
\]
\end{defn}
The following lemma is known from \cite{HIMN}, where it was proved for $CK_*$. One can easily check that the proof works for an arbitrary oriented Borel-Moore homology and in particular for $\Omega_*$.
\begin{lem}\label{Ylem2C}
For each $j=1,\dots, r$, the variety $Y_j$ is regularly embedded in $P_{j-1}$ and $P_{j-1}$ is regularly embedded in $\PP(U/D_{j-1})$. Furthermore, in $\Omega_*(P_{j-1})$, we have
\[
\iota_{j*}(1_{Y_j})=\tilde{c}_{\lambda_j + n - k - j}\Big((D_j/D_{j-1})^{\vee}\otimes (D_{\gamma_j}^{\perp}/F^{\chi_j})\Big)(1_{P_{j-1}}).
\]
\end{lem}
\begin{defn}\label{dfkappaC}
Let $\sfp:=\sfp_1\circ\cdots\circ\sfp_r: Y_{\lambda}^C \to \SG^k(E)$. Define the class $\kappa_{\lambda}^C \in \Omega_*(\SG^k(E))$ by
\[
\kappa_{\lambda}^C:= \sfp_*(1_{Y_{\lambda}^C}).
\]
\end{defn}
\begin{rem}
It is also known that $Y_{\lambda}^C$ is irreducible and has at worst rational singularity. Furthermore $Y_{\lambda}^C$ is birational to $X_{\lambda}^C$ through the projection $\pi$ (See for example, \cite{HIMN}). Therefore in $K$-theory and Chow ring of $\SG^k(E)$ the class $\kappa_{\lambda}^C$ coincides with the fundamental class of the degeneracy loci $X_{\lambda}^C$. Note that in a general oriented cohomology theory, the fundamental class of $X_{\lambda}^C$ is not defined since $X_{\lambda}^C$ may not be an l.c.i scheme.
\end{rem}
\subsection{Computing $\kappa_{\lambda}^C$}
\begin{defn}\label{dfCclass}
For each $m\in \ZZ$ and $-n\leq \ell \leq n$, define 
\[
\tscC_m^{(\ell)} := \tscS_{m}\big(U^\vee-(E/F^{\ell})^{\vee}\big).
\]
In $\Omega^*(\SG^k(E))$, we denote $\scC_m^{(\ell)} := \tscC_m^{(\ell)} (1_{\SG^k(E)})$.
\end{defn}
\begin{lem}\label{lempushC}
In $\Omega_*(Y_{j-1})$, we have
\begin{equation}\label{eqp1}
\sfp_{j*}\circ  \tilde{\tau}_j^s(1_{Y_j}) = \sum_{p=0}^{j-1}\sum_{q=0}^{\infty}(-1)^p\tilde{c}_p(D_{j-1}^{\vee} - D_{\gamma_j}) \circ{w}_{-q}(D_{j-1}^{\vee} - D_{\gamma_j})\circ\tscC_{\lambda_j+s-p+q}^{(\chi_j)}(1_{Y_{j-1}}).
\end{equation}
\end{lem}
\begin{proof}
By Lemma \ref{Ylem2C}, we have
\begin{eqnarray*}
\sfp_{j*}\circ  \tilde{\tau}_j^s(1_{Y_j}) 
=\pi'_{j*}\circ\iota_{j*}\circ  \tilde{\tau}_j^s(1_{Y_j}) 
=\pi'_{j*}\circ  \tilde{\tau}_j^s\circ\iota_{j*}(1_{Y_j}) 
=\pi'_{j*}\circ  \tilde{\tau}_j^s\circ\tilde{\alpha}_j(1_{P_{j-1}}),
\end{eqnarray*}
where $\tilde{\alpha}_j:=\tilde{c}_{\lambda_j + n - k - j}((D_j/D_{j-1})^{\vee}\otimes (D_{\gamma_j}^{\perp}/F^{\chi_j}))$.  By Theorem \ref{thmlcipush}, we have
\begin{eqnarray*}
\pi_{j*}'\circ \tilde{\tau}_j^s \circ \tilde{\alpha}_j(1_{P_{j-1}}) 
&=& \tscS_{s+\lambda_j}\big((U/D_{j-1})^{\vee} - (D_{\gamma_j}^{\perp}/F^{\chi_j})^{\vee}\big)(1_{Y_{j-1}})\\
&=& \tscS_{s+\lambda_j}\big(U^{\vee}- (E/F^{\chi_j})^{\vee} - (D_{j-1}- D_{\gamma_j}^{\vee})^{\vee}\big)(1_{Y_{j-1}}),
\end{eqnarray*}
where we have used $D_{\gamma_j}^{\perp} = E - D_{\gamma_j}^{\vee}$. Now the claim follows from (\ref{eqRelSeg1}).
\end{proof}


Before we continue, we need to establish some notation. Let $R=\Omega^*(\Gr_d(E))$, viewed as a graded algebra over $\LL$, and let $t_1,\ldots,t_{r}$ be indeterminates of degree $1$. We use the multi-index notation $t^\bfs:=t_1^{s_1}\cdots t_{r}^{s_{r}}$ for $\bfs=(s_1,\dots,s_{r})\in \ZZ^{r}$. A formal Laurent series $f(t_1,\ldots,t_{r})=\sum_{\bfs\in\ZZ^{r}}a_{\bfs}t^{\bfs}$ is {\em homogeneous of degree} $m\in \ZZ$ if $a_{\bfs}$ is zero unless $a_{\bfs}\in R_{m-|\bfs|}$ with $|\bfs|=\sum_{i=1}^{r} s_i$. Let $\supp\, f = \{\bfs \in \ZZ^r \ |\ a_{\bfs}\not=0\}$.
For each $m \in \ZZ$, define $\calL^{R}_m$ to be the space of all formal Laurent series of homogeneous degree $m$ such that there exists $\sfn\in \ZZ^r$ for which $\sfn + \supp\, f$ is contained in the cone in $\ZZ^r$ defined by $s_1\geq0, \; s_1+s_2\geq 0, \;\cdots, \; s_1+\cdots + s_{r} \geq 0$. Then $\calL^{R}:=\bigoplus_{m\in \ZZ} \calL^{R}_m$ is a graded ring over $R$ with the obvious product. 
For each $i=1,\dots, r$, let $\calL^{R,i}$ be the $R$-subring of $\calL^R$ consisting of series that do not contain any negative powers of $ t_1,\dots, t_{i-1}$.  In particular, $\calL^{R,1}=\calL^{R}$. 
A series $f(t_1,\ldots,t_{r})$ is a {\em power series} if it doesn't contain any negative powers of $t_1,\dots,t_r$. Let $R[[t_1,\ldots,t_r]]_{m}$ denote the set of all power series in $t_1,\dots, t_r$ of degree $m\in \ZZ$. We set $R[[t_1,\ldots,t_r]]_{\gr}:=\bigoplus_{m\in \ZZ}R[[t_1,\ldots,t_r]]_{m}$.

 We are now able to define the substitution morphism $\phi_j$.
\begin{defn}
Define a graded $R$-module homomorphism $\phi_1: \calL^{R} \to \Omega_*(\SG^k(E))$ as 
\[
\phi_1^C( t_1^{s_1}\cdots  t_r^{s_r})= \tscC_{s_1}^{(\chi_1)} \circ\cdots \circ\tscC_{s_r}^{(\chi_r)}(1_{\SG^k(E)}).
\]  
Similarly, for each $j=2,\dots, d$, define a graded $R$-module homomorphism $\phi_{j}^C: \calL^{R,j} \to \Omega_*(Y_{j-1})$ by setting
\[
\phi_j^C( t_1^{s_1}\cdots  t_r^{s_r})= \tilde{\tau}_1^{s_1}\circ\cdots  \circ\tilde{\tau}_{j-1}^{s_{j-1}}\circ\tscC_{s_j}^{(\chi_j)} \circ\cdots \circ\tscC_{s_r}^{(\chi_r)}(1_{Y_{j-1}}).
\]
\end{defn}
\begin{rem}\label{remHomotoCoho}
By regarding $\Omega^*(\SG^k(E))=\Omega_{\dim \SG^k(E) - *}(\SG^k(E))$, we have
\[
\phi^C_1( t_1^{s_1}\cdots  t_r^{s_r})= \scC_{s_1}^{(\chi_1)} \cdots \scC_{s_r}^{(\chi_r)}.
\]  
\end{rem}
Using $\phi_j^C$, we can restate (\ref{eqp1}) as follows.
\begin{lem}\label{pphiC}
One has
\[
\sfp_{j*}\circ\tilde{\tau}_j^s(1_{Y_j})=\phi_j^C\left(t_j^{\lambda_j+s}\frac{\prod_{i=1}^{j-1}(1 - t_i/t_j)P(t_j, t_i)}{\prod_{i=1}^{\gamma_j}(1 - \bar t_i/t_j)P(t_j, \bar t_i)} \right).
\]
\end{lem}
\begin{proof}
Consider the functions of $t_1,\dots, t_{j-1}$ defined by the following generating functions:
\begin{eqnarray*}
\sum_{p=0}^{\infty} H_p^{\lambda}(t_1,\dots,t_{j-1})u^p &:=& \frac{e(t_1,\dots, t_{j-1}; u)}{e(\bar t_1,\dots, \bar t_{\gamma_j}; u)} 
=\frac{\prod_{i=1}^{j-1}(1 + t_i u)}{\prod_{i=1}^{\gamma_j}(1 + \bar t_i u)}\\
\sum_{q = 0}^{\infty} W_{-q}^{\lambda}(t_1,\dots,t_{j-1}) u^{-q} &:=& \frac{w(t_1,\dots, t_{j-1}; u)}{w(\bar t_1,\dots, \bar t_{\gamma_j}; u)} 
=  \frac{\prod_{i =1}^{j-1}P(u^{-1}, t_i)}{\prod_{i =1}^{\gamma_j}P(u^{-1}, \bar t_i)}.
\end{eqnarray*}
Then we have
\[
H_p^{\lambda}(\tilde{\tau}_1,\dots,\tilde{\tau}_{j-1}) = \tilde{c}_p(D_{j-1}^{\vee} - D_{\gamma_j}), \ \ \ W_{-q}^{\lambda}(\tilde{\tau}_1,\dots,\tilde{\tau}_{j-1})=\tilde{w}_{-q}(D_{j-1}^{\vee} - D_{\gamma_j}).
\]
Thus,  by (\ref{eqp1}) and the definition of $\phi_j^C$, we have
\begin{eqnarray*}
\sfp_{j*}\circ\tilde{\tau}_j^s(1_{Y_j})
&=&\phi_j^C\left(\sum_{p=0}^{j-1}\sum_{q=0}^{\infty}(-1)^pH_p^{\lambda}(t_1,\dots,t_{j-1}) W_{-q}^{\lambda}(t_1,\dots,t_{j-1})  t_j^{\lambda_j+s-p+q}\right)\\
&=&\phi_j^C\left(t_j^{\lambda_j+s}\left(\sum_{p=0}^{j-1}(-1)^pH_p^{\lambda}(t_1,\dots,t_{j-1})t_j^{-p} \right)\left(\sum_{q=0}^{\infty}W_{-q}^{\lambda}(t_1,\dots,t_{j-1})t_j^q\right)  \right).
\end{eqnarray*}
The claim follows from the definitions of $H_p^{\lambda}$ and $W_{-q}^{\lambda}$ in terms of the generating functions.
\end{proof}
Finally, we are able to prove the main theorem in the case of symplectic Grassmann bundles.
\begin{thm}\label{detthmC}
For a strict partition $\lambda \in \SP^k(n)$, the associated symplectic Kempf--Laksov class $\kappa^C_{\lambda}$ is given by
\[
\kappa^C_{\lambda} = \sum_{\sfs=(s_1,\dots, s_r) \in\ZZ^r}f_{\sfs}^{\lambda}\scC_{s_1+\lambda_1}^{(\chi_1)} \cdots \scC_{s_r+\lambda_r}^{(\chi_r)}.
\]
where $f_{\sfs}^{\lambda}\in \LL$ are the coefficients of the Laurent series 
\begin{equation}\label{notationFofBC}
\frac{\prod_{1\leq i<j\leq r}(1-t_i/t_j)P(t_j,t_i)}{\prod_{(i,j)\in C(\lambda)}(1-\bar t_i/t_j)P(t_j, \bar t_i)}=\sum_{\sfs=(s_1,\dots, s_r) \in\ZZ^r}f_{\sfs}^{\lambda}\cdot t_1^{s_1}\cdots t_r^{s_r}
\end{equation}
as an element of $\calL^{\LL}$.
\end{thm}
\begin{proof}
By Definition \ref{dfkappaC}, it follows from successive applications of Lemma \ref{pphiC} (\textit{cf.} \cite{HIMN}) that
\[
\kappa^C_{\lambda} = \phi_1^C\left(t_1^{\lambda_1}\cdots t_r^{\lambda_r}  \frac{\prod_{1\leq i<j\leq r}(1-t_i/t_j)P(t_j,t_i)}{\prod_{(i,j)\in C(\lambda)}(1-\bar t_i/t_j)P(t_j, \bar t_i)}  \right).
\]
Then, in view of the definition of the coefficients $f_{\sfs}$, it suffices to apply $\phi_1^C$ to obtain the claim.
\end{proof}
\section{Odd orthogonal degeneracy loci}

For this section we fix a nonnegative integer $k$.
\subsection{Degeneracy loci}\label{secKLB}
Consider the vector bundle $E$ of rank $2n+1$ over $X$ with a symmetric non-degenerate bilinear form $\lan \ ,\  \ran: E \otimes E \to \calO_X$ where $\calO_X$ is the trivial line bundle over $X$. Let $\xi: \OG^k(E) \to X$ be the Grassmann bundle consisting of pairs $(x,U_x)$ where $x\in X$ and $U_x$ is an $n-k$ dimensional isotropic subspace of $E_x$. Note that the bilinear form $\lan\ ,\ \ran$ on $E$ induces an isomorphism  $F^{\perp}/F \otimes F^{\perp}/F \cong \calO_X$ for any maximal isotropic subbundle $F$ of $E$ where $F^{\perp}$ is the orthogonal complement of $F$ with respect to $\lan\ ,\ \ran$. This implies that $c_1(F^{\perp}/F)=0$ in  $\Omega^*(X)\otimes_{\ZZ}\ZZ[1/2]$.

Fix a reference flag 
\[
0=F^n\subset F^{n-1} \subset \cdots \subset F^1 \subset F^0 \subset (F^0)^{\perp} \subset F^{-1} \subset \cdots \subset F^{-n}=E,
\]
such that $\rk F^i=n-i$ for $i\geq 0$ and $(F^{i})^{\perp}=F^{-i}$ for all $i \geq 1$.  For each $\lambda \in \SP^k(n)$ of length $r$, we define the associated degeneracy loci $X_{\lambda}^B$ in $\OG^k(E)$ is defined by
\[
X_{\lambda}^B = \Big\{(x,U_x)\in \OG^k(E) \ |\ \dim(U_x \cap F^{\chi_i})\geq i, \ i=1,\dots,r\Big\}.
\]
where $\chi$ is the characteristic index associated to $\lambda$.
\subsection{Quadric bundle}
The bundle $\OG^{n-1}(E)$ is also known as the quadric bundle and we denote it by $Q(E)$. In this section, we do not assume that $X$ is smooth as long as it is regularly embedded in a quasi-projective smooth variety. Let $S$ be the tautological line bundle of $Q(E)$. In this particular case the Schubert varieties of $Q(E)$ are indexed by a single integer $\lambda_1$ and can be explicitly described as follows:
\begin{equation}\label{Xi}
X^B_{\lambda_1}=\begin{cases}
Q(E)\cap\PP(F^{\lambda_1-n}) &\  (0\leq \lambda_1<n)\\
\PP(F^{\lambda_1-n}) & (n\leq \lambda_1< 2n).
\end{cases}
\end{equation}
It is worth noticing that $\lambda_1$ represents the codimension of $X^B_{\lambda_1}$ in $Q(E)$.

\begin{lem}
The fundamental class of the subvariety $X^B_{\lambda_1}$ in $\Omega_*(Q(E))$ for $\lambda_1<n$ is given by 
\begin{equation}\label{Xiformula1}
[X^B_{\lambda_1} \to Q(E)] = \tilde{c}_{\lambda_1}(S^{\vee}\otimes E/F^{\lambda_1-n})(1_{Q(E)}).
\end{equation}
Moreover the fundamental class of $X^B_{\lambda_1}$ in $\Omega_*(Q(E))$ for $\lambda_1\geq n$ satisfies the following identity
\begin{equation}\label{Xiformula2}
F^{(2)}\Big(\tilde{c}_1(S^{\vee}\otimes (F^0)^{\perp}/F^0)\Big)\big([X^B_{\lambda_1} \to Q(E)]\big) = \tilde{c}_{\lambda_1}\Big(S^{\vee}\otimes (E/(F^0)^{\perp}\oplus F^0/F^{\lambda_1-n})\Big)\big(1_{Q(E)}\big).
\end{equation}
where $F^{(2)}$ is a special case of the power series defined in (\ref{formalmult}).
\end{lem}
\begin{proof}
The formula (\ref{Xiformula1}) follows from Lemma \ref{lemCKGB}. For (\ref{Xiformula2}), first we show the case for $\lambda_1=n$, by computing the class $[X^B_n\to Q(E)]$ in $\Omega^*(Q(E))$ in two different ways. The variety $X^B_n$ is a divisor in $\PP((F^0)^{\perp})$, corresponding the line bundle $S^{\vee}\otimes (F^0)^{\perp}/F^0$. Moreover, the scheme theoretic intersection $Q(E)\cap \PP((F^0)^{\perp})$ defines the Weil divisor $2 X^B_n$ on $\PP((F^0)^{\perp})$ and in view of Lemma \ref{lem div} we have 
\[
1_{Q(E) \cap \PP((F^0)^{\perp})}=\iota_* \Big(F^{(2)}\big(\tilde{c} _1(S^{\vee}\otimes (F^0)^{\perp}/F^0)\big)(1_{X^B_n})\Big),
\]
where $\iota: X^B_n \to Q(E)\cap \PP((F^0)^{\perp})$ is the obvious inclusion. Thus, by pushing forward this identity to $Q(E)$, we obtain the following identity in $\Omega_*(Q(E))$:
\[
[Q(E) \cap \PP((F^0)^{\perp}) \to Q(E)]=F^{(2)}\Big(\tilde{c} _1(S^{\vee}\otimes F^0)^{\perp}/F^0\Big)\big([X^B_n \to Q(E)]\big).
\]
On the other hand, Lemma \ref{lemCKGB} implies that 
\[
[Q(E)\cap \PP((F^0)^{\perp}) \to Q(E)]=\tilde{c}_n\Big(S^{\vee}\otimes E/(F^0)^{\perp}\Big).
\]
This proves (\ref{Xiformula2}) for $\lambda_1=n$. 

If $\lambda_1>n$, by Lemma \ref{lemCKGB} we have $[X^B_{\lambda_1} \to X^B_n] = \tilde{c}_{i}(S^{\vee}\otimes F^0/F^i)$ in $\Omega_*(X^B_n)$. Thus by pushing forward the identity
\[
F^{(2)}\Big(\tilde{c} _1(S^{\vee}\otimes (F^0)^{\perp}/F^0\Big)\big([X^B_{\lambda_1} \to X^B_n]\big) =  F^{(2)}\Big(\tilde{c}_1(S^{\vee}\otimes F^0)^{\perp}/F^0\Big)\circ\tilde{c}_{i}(S^{\vee}\otimes F^0/F^i)(1_{X^B_n})
\]
to $\Omega^*(Q(E))$, and applying (\ref{Xiformula2}) for $\lambda_1=n$, we obtain
\begin{eqnarray*}
F^{(2)}\Big(\tilde{c} _1(S^{\vee}\otimes (F^0)^{\perp}/F^0\Big)\big([X^B_{\lambda_1} \to Q(E)]\big)&=& \tilde{c}_n\Big(S^{\vee}\otimes E/(F^0)^{\perp}\Big)\circ\tilde{c}_{\lambda_1-n}(S^{\vee}\otimes F^0/F^{\lambda_1-n})(1_{Q(E)})\\
&=& \tilde{c}_{\lambda_1}\Big(S^{\vee}\otimes (E/(F^0)^{\perp}\oplus F^0/F^i)\Big)(1_{Q(E)}).
\end{eqnarray*}
Thus (\ref{Xiformula2}) holds.
\end{proof}
As mentioned above, we have $\tilde{c}_1((F^0)^{\perp}/F^0)=0$ in $\Omega_*(Q(E))\otimes_{\ZZ} \ZZ[1/2]$ so that $\tilde{c}_1(S^{\vee}\otimes (F^0)^{\perp}/F^0) = \tilde{c}_1(S^{\vee})$. Therefore we have
\[
F^{(2)}\Big(\tilde{c}_1(S^{\vee}\otimes (F^0)^{\perp}/F^0)\Big) = F^{(2)}\Big(\tilde{c}_1(S^{\vee})\Big).
\]
Notice that, since it is homogeneous of degree $0$ with constant term $2$, the series $F^{(2)}(u)$ is invertible in $\LL\otimes_{\ZZ}\ZZ[1/2]$. Thus we have the following corollary.
\begin{cor}\label{corXi}
In $\Omega_*(Q(E))\otimes_{\ZZ}\ZZ[1/2]$, we have 
\[
[X^B_{\lambda_1}\to Q(E)] = \begin{cases}
\tilde{c}_{\lambda_1}(S^{\vee}\otimes E/F^{\lambda_1-n})(1_{Q(E)}) & (0\leq \lambda_1 < n)\\
\displaystyle\frac{1}{F^{(2)}\big(\tilde{c}_1(S^{\vee})\big)}\circ \tilde{c}_{\lambda_1}(S^{\vee}\otimes E/F^{\lambda_1-n})(1_{Q(E)}) & (n\leq  \lambda_1 < 2n).
\end{cases}
\]
\end{cor}
\begin{rem}
As mentioned in Remark \ref{remLINE}, we have
\[
[X^B_{\lambda_1} \to Q(E)] = \begin{cases}
\tscS_{\lambda_1}\big(S^\vee- (E/F^{\lambda_1-n})^{\vee}\big)(1_{Q(E)}) &\  (0\leq \lambda_1 < n)\\
 \displaystyle \frac{1}{F^{(2)}\big(c_1(S^\vee)\big)} \tscS_{\lambda_1}\big(S^\vee- (E/F^{\lambda_1-n})^{\vee}\big)(1_{Q(E)}) & (n\leq  \lambda_1 \leq 2n).
\end{cases}
\]
\end{rem}
\subsection{Resolution of singularities}\label{secBres}
Consider the $r$-step flag bundle $\pi: \Fl_r(U) \to \OG^k(E)$ as before. We let $D_1\subset \cdots \subset D_r$ be the tautological flag. Recall that $\Fl_r(U)$ can be constructed as the tower of projective bundles
\begin{equation}\label{BPtower}
\pi: \Fl_r(U)=\PP(U/D_{r-1}) \stackrel{\pi_r}{\to} \cdots \stackrel{\pi_{3}}{\to}\PP(U/D_1) \stackrel{\pi_2}{\to}  \PP(U) \stackrel{\pi_1}{\to} \OG^k(E)
\end{equation}
We regard $D_j/D_{j-1}$ as the tautological line bundle of $\PP(U/D_{j-1})$ where we let $D_0=0$. For each $j=1,\dots,r$, let $\tilde{\tau}_j:=\tilde{c}_1((D_j/D_{j-1})^{\vee})$ be the firstChern class operator of $(D_j/D_{j-1})^{\vee}$ on $\CK_*(\PP(U/D_{j-1}))$.
\begin{defn}
For each $j=1,\dots,r$, we define a subvariety $Y_j$ of $\PP(U/D_{j-1})$ by setting
\[
Y_j := \Big\{ (x,U_x, (D_1)_x, \dots, (D_j)_x) \in \PP(U/D_{j-1}) \ |\ (D_i)_x \subset F^{\chi_i}_x, \ i=1,\dots,{j}\Big\}.
\]
We set $Y_0:=\SG^k_r(U)$ and $Y^B_{\lambda}:=Y_r$. Let $P_{j-1}:=\pi_j^{-1}(Y_{j-1})$, $\pi_j': P_{j-1} \to Y_{j-1}$ the projection and $\iota_j: Y_j \to P_{j-1}$ the obvious inclusion. Let $\sfp_j:=\pi_j'\circ\iota_j$. We have the commutative diagram
\[
\xymatrix{
\PP(U/D_{j-1}) \ar[r]_{\pi_j} & \PP(U/D_{j-2})\\
P_{j-1} \ar[r]_{\pi_j'}\ar[u] & Y_{j-1}\ar[u]\\
Y_j\ar[u]_{\iota_j}\ar[ru]_{\sfp_j} &
}
\]
As in the symplectic case we set $\sfp:=\sfp_1\circ\cdots\circ\sfp_r: Y^B_{\lambda} \to \OG^k(E)$ and define
$$
\kappa_{\lambda}^B:= \sfp_*(1_{Y_{\lambda}^B}).
$$
\end{defn}
The following lemma is known from \cite{HIMN}, where the computation of the fundamental class of $Y_j$ in $P_{j-1}$ is done in connective $K$-theory $\CK_*$. However, the proof is valid in an arbitrary oriented Borel-Moore homology and in particular in $\Omega_*$.
\begin{lem}\label{Ylem2B}
For each $j=1,\dots, r$, the variety $Y_j$ is regularly embedded in $P_{j-1}$ and $P_{j-1}$ is regularly embedded in $\PP(U/D_{j-1})$, in particular they both belong to $\LCI$. Moreover we have
\[
\iota_{j*}(1_{Y_j})=\tilde{\alpha}_j(1_{P_{j-1}}),
\]
in $\CK_*(P_{j-1})$, where
\[
\tilde{\alpha}_j = \begin{cases}
\tilde{c}_{\lambda_j + n - k - j}\big((D_j/D_{j-1})^{\vee}\otimes (D_{\gamma_j}^{\perp}/F^{\chi_j})\big) & (-n\leq \chi_j<0)\\
\displaystyle \frac{1}{F^{(2)}\Big(\big(c_1(D_j/D_{j-1})^\vee\big)\Big)} \tilde{c}_{\lambda_j + n - k - j}\big((D_j/D_{j-1})^{\vee}\otimes (D_{\gamma_j}^{\perp}/F^{\chi_j})\big)
& \ (0\leq \chi_j< n).
\end{cases}
\]
\end{lem}
\subsection{Computing $\kappa_{\lambda}^B$}\label{secPFthm}
\begin{defn}\label{dfBclass}
Let $-n\leq \ell <n$. For each $m\in \ZZ$, we define the operators $\scB_m^{(\ell)}$ for $\Omega_*(\OG^k(E))\otimes_{\ZZ}\ZZ[1/2]$ by means of the following generating function
\[
\sum_{m\in \ZZ}\tscB_m^{(\ell)} u^m = 
\begin{cases}
\tscS\big(U^\vee- (E/F^{\ell})^{\vee};u\big) & (-n\leq \ell <0)\\
\displaystyle\frac{1}{F^{(2)}(u^{-1})} \tscS\big(U^\vee- (E/F^{\ell})^{\vee};u\big) &\  (0\leq \ell < n).
\end{cases}
\]
If $F^{(2)}(u^{-1})^{-1} = \sum_{s\geq 0} f_s u^{-s}$ with $f_s\in \LL\otimes_{\ZZ}\ZZ[1/2]$, then for each $\ell\geq 0$, we have 
\[
\tscB_m^{(\ell)} = \sum_{s\geq 0}f_s \tscS_{m+s}\big(U^\vee- (E/F^{\ell})^{\vee}\big).
\]
\end{defn}
\begin{rem}
If $\lambda=(\lambda_1) \in \SP^k(n)$, we have $\kappa_{\lambda}^B = \scB_{\lambda_1}^{(\chi_1)}$.
\end{rem}
\begin{lem}\label{rth first B} 
For each $s\geq 0$, we have
\begin{eqnarray}\label{eqstageB}
\sfp_{j*}\circ\tilde{\tau}_j^s(1_{Y_j})&=& \sum_{p=0}^{\infty}\sum_{q=0}^{\infty}(-1)^p\tilde{c}_p(D_{j-1}^{\vee} - D_{\gamma_j}) \circ\tilde{w}_{-q}(D_{j-1}^{\vee} - D_{\gamma_j})\circ\tscB_{\lambda_j+s-p+q}^{(\chi_j)}(1_{Y_{j-1}}).
\end{eqnarray}
\end{lem}
\begin{proof}
By Lemma \ref{Ylem2B}, we have
\begin{eqnarray*}
\sfp_{j*}\circ  \tilde{\tau}_j^s(1_{Y_j}) 
=\pi'_{j*}\circ\iota_{j*}\circ  \tilde{\tau}_j^s(1_{Y_j}) 
=\pi'_{j*}\circ  \tilde{\tau}_j^s\circ\iota_{j*}(1_{Y_j}) 
=\pi'_{j*}\circ  \tilde{\tau}_j^s\circ\tilde{\alpha}_j(1_{P_{j-1}}).
\end{eqnarray*}
Suppose that $\chi_j<0$. By Theorem \ref{thmlcipush},  the right hand side equals 
\[
\tscS_{\lambda_j+s}\Big((U/D_{j-1} - D_{\gamma_j}^{\perp}/F^{\chi_j})^{\vee}\Big)(1_{Y_{j-1}})
=\tscS_{\lambda_j+s}\Big((U- E/F^{\chi_j} -D_{j-1} + D_{\gamma_j}^{\vee})^{\vee}\Big)(1_{Y_{j-1}})
\]
where $D_{\gamma_j}^{\perp} = E - D_{\gamma_j}^{\vee}$. Then the claim follows from (\ref{eqRelSeg1}). Similarly, if $0\leq \chi_j$, Theorem \ref{thmlcipush} implies that the right hand side equals 
\[
\sum_{s'=0}^{\infty}f_{s'}\tscS_{\lambda_j+s+s'}\Big((U/D_{j-1})^\vee - (D_{\gamma_j}^{\perp}/F^{\chi_j})^{\vee}\Big)(1_{Y_{j-1}})
\]
where we set $F^{(2)}(u^{-1})^{-1} = \sum_{s'\geq 0} f_{s'} u^{-s'}$ with $f_{s'}\in \LL\otimes_{\ZZ}[1/2]$ as above. Again, we use the identity $D_{\gamma_j}^{\perp} = E - D_{\gamma_j}^{\vee}$ and then the claim follows from (\ref{eqRelSeg1}).
\end{proof}
Set $R:=\Omega^*(\OG^k(E))\otimes_{\ZZ}\ZZ[1/2]$ and let $\calL^R$ be the ring of formal Laurent series with indeterminates $t_1,\dots, t_{r}$ defined in the previous section.
\begin{defn}
Define a graded $R$-module homomorphism $\phi_1^B: \calL^{R}\otimes_{\ZZ}\ZZ[1/2] \to \Omega_*(\OG^k(E))\otimes_{\ZZ}\ZZ[1/2]$ by 
\[
\phi_1^B(t_1^{s_1}\cdots  t_r^{s_r})= \tscS_{s_1}\Big(U^\vee -(E/F^{\chi_1})^\vee\Big) \circ\cdots\circ  \tscS_{s_r}\Big(U^\vee -(E/F^{\chi_r})^\vee\Big)(1_{\OG^k(E)}).
\]
Similarly, for each $j=2,\dots, d$, define a graded $R$-module homomorphism $\phi_{j}^B: \calL^{R,j}\otimes_{\ZZ}\ZZ[1/2] \to \Omega_*(Y_{j-1})\otimes_{\ZZ}\ZZ[1/2]$ by
\[
\phi_j^B( t_1^{s_1}\cdots  t_r^{s_r})=  \tilde{\tau}_1^{s_1}\circ\cdots \circ\tilde{\tau}_{j-1}^{s_j}\circ\tscS_{s_j}\Big(U^\vee -(E/F^{\chi_j})^\vee\Big)\circ \cdots  \circ\tscS_{s_r}\Big(U^\vee -(E/F^{\chi_r})^\vee\Big)(1_{Y_{j-1}}).
\]
Note that for each $i$ such that $j\leq i \leq r$ and $\chi_i\geq 0$, we have 
\begin{equation}\label{remP}
\phi_j^B\left(\frac{t_i^{m}}{F^{(2)}(t_i)}\right) =\tscB_m^{(\chi_i)}(1_{Y_{j-1}}), \ \ \ m\in \ZZ.
\end{equation}
\end{defn}
As with Lemma \ref{pphiC}, by making use of Lemma \ref{rth first B} we can prove the following lemma.
\begin{lem}\label{MainLemB}
We have
\[
\sfp_{j*}\circ  \tilde{\tau}_j^s(1_{Y_j})= \begin{cases}
\phi_j^B\left(  t_j^{\lambda_j+s}\dfrac{\prod_{i=1}^{j-1}(1 - t_i/t_j)P(t_j, t_i)}{\prod_{i=1}^{\gamma_j}(1 - \bar t_i/t_j)P(t_j, \bar t_i)} \right) & (\chi_j<0)\\
\phi_j^B\left(  \dfrac{t_j^{\lambda_j+s}}{F^{(2)}(t_j)}\dfrac{\prod_{i=1}^{j-1}(1 - t_i/t_j)P(t_j, t_i)}{\prod_{i=1}^{\gamma_j}(1 - \bar t_i/t_j)P(t_j, \bar t_i)}\right) & (0\leq \chi_j).
\end{cases}
\]
for all $s\geq 0$.
\end{lem}
A repeated application of Lemma \ref{MainLemB} to the definition of $\kappa_{\lambda}^B$, together with (\ref{remP}), allows us to obtain the main theorem for odd orthogonal Grassmannians.
\begin{thm}\label{MainThmB} 
We have
\[
\kappa_{\lambda}^B 
=\sum_{\sfs=(s_1,\dots, s_r) \in\ZZ^r}c_{\sfs}^{\lambda}\scB_{\lambda_1+s_1}^{(\chi_1)} \cdots \scB_{\lambda_r+s_r}^{(\chi_r)}.
\]
where the $f_{\sfs}^{\lambda}\in \LL$ are the coefficients of the Laurent series
\begin{equation}
\frac{\prod_{1\leq i<j\leq r}(1-t_i/t_j)P(t_j,t_i)}{\prod_{(i,j)\in C(\lambda)}(1-\bar t_i/t_j)P(t_j, \bar t_i)}=\sum_{\sfs=(s_1,\dots, s_r) \in\ZZ^r}c_{\sfs}^{\lambda}\cdot t_1^{s_1}\cdots t_r^{s_r}
\end{equation}
as an element of $\calL^{\LL}$.
\end{thm}

\bibliography{references}{}

\begin{thebibliography}{10}

\bibitem{K-Anderson}
{\sc {Anderson}, D.}
\newblock {K-theoretic Chern class formulas for vexillary degeneracy loci}.
\newblock {\em ArXiv e-prints\/} (Dec. 2017).

\bibitem{AndersonFulton}
{\sc {Anderson}, D., and {Fulton}, W.}
\newblock {Degeneracy Loci, Pfaffians, and Vexillary Signed Permutations in
  Types B, C, and D}.
\newblock {\em ArXiv e-prints\/} (Oct. 2012).

\bibitem{AndersonFulton2}
{\sc {Anderson}, D., and {Fulton}, W.}
\newblock {Chern class formulas for classical-type degeneracy loci}.
\newblock {\em ArXiv e-prints\/} (Apr. 2015).

\bibitem{BuchKreschTamvakis1}
{\sc Buch, A.~S., Kresch, A., and Tamvakis, H.}
\newblock A {G}iambelli formula for isotropic {G}rassmannians.
\newblock {\em Selecta Math. (N.S.) 23}, 2 (2017), 869--914.

\bibitem{IntersectionFulton}
{\sc Fulton, W.}
\newblock {\em Intersection theory}, second~ed., vol.~2 of {\em Ergebnisse der
  Mathematik und ihrer Grenzgebiete. 3. Folge. A Series of Modern Surveys in
  Mathematics [Results in Mathematics and Related Areas. 3rd Series. A Series
  of Modern Surveys in Mathematics]}.
\newblock Springer-Verlag, Berlin, 1998.

\bibitem{HIMN}
{\sc {Hudson}, T., {Ikeda}, T., {Matsumura}, T., and {Naruse}, H.}
\newblock Degeneracy loci classes in k-theory — determinantal and pfaffian
  formula.
\newblock {\em Advances in Mathematics 320\/} (2017), 115 -- 156.

\bibitem{HudsonMatsumura}
{\sc {Hudson}, T., and {Matsumura}, T.}
\newblock {Segre classes and Damon--Kempf--Laksov formula in algebraic
  cobordism}.
\newblock {\em ArXiv e-prints\/} (Feb. 2016).

\bibitem{Ikeda}
{\sc Ikeda, T.}
\newblock Schubert classes in the equivariant cohomology of the {L}agrangian
  {G}rassmannian.
\newblock {\em Adv. Math. 215}, 1 (2007), 1--23.

\bibitem{IkedaMatsumura}
{\sc Ikeda, T., and Matsumura, T.}
\newblock Pfaffian sum formula for the symplectic {G}rassmannians.
\newblock {\em Math. Z. 280}, 1-2 (2015), 269--306.

\bibitem{Kazarian}
{\sc Kazarian, M.}
\newblock On lagrange and symmetric degeneracy loci.
\newblock {\em Isaac Newton Institute for Mathematical Sciences Preprint
  Series\/} (2000).

\bibitem{LevineMorel}
{\sc Levine, M., and Morel, F.}
\newblock {\em Algebraic cobordism}.
\newblock Springer Monographs in Mathematics. Springer, Berlin, 2007.

\bibitem{PragaczPQ}
{\sc Pragacz, P.}
\newblock Algebro-geometric applications of {S}chur {$S$}- and
  {$Q$}-polynomials.
\newblock In {\em Topics in invariant theory ({P}aris, 1989/1990)}, vol.~1478
  of {\em Lecture Notes in Math.} Springer, Berlin, 1991, pp.~130--191.

\bibitem{TamvakisWilson}
{\sc Tamvakis, H., and Wilson, E.}
\newblock Double theta polynomials and equivariant {G}iambelli formulas.
\newblock {\em Math. Proc. Cambridge Philos. Soc. 160}, 2 (2016), 353--377.

\bibitem{WilsonThesis}
{\sc Wilson, V.}
\newblock {E}quivariant {G}iambelli {F}ormulae for {G}rassmannians.
\newblock Ph.D. thesis. University of Maryland (2010).

\end{thebibliography}
\bibliographystyle{acm}
\end{document}